\documentclass[11pt]{article}
\usepackage{epsfig}
\usepackage{amssymb,amsmath, amscd, amsfonts, amssymb, graphicx, color}
\usepackage{caption}
\usepackage{cancel}
\usepackage{tikz}
\newtheorem{theorem}{{\bf Theorem}}[section]
\newtheorem{thm}{{\bf Theorem}}[section]

\newtheorem{lemma}[thm]{\bf Lemma}

\newtheorem{corollary}[thm]{\bf Corollary}
\newtheorem{proposition}[thm]{\bf Proposition}

\newtheorem{preproof}{{\bf Proof.}}

\newenvironment{proof}[1]{\begin{preproof}{\rm
               #1}\hfill{$\rule{2mm}{2mm}$}}{\end{preproof}}
\begin{document}
\title{\Large {\bf Total dominator chromatic number of Kneser graphs}}

\author{
{ Parvin Jalilolghadr\thanks{p$\_$jalilolghadr@yahoo.com}}   \\
[1mm]
{\it \small Department of  Mathematics,  University of Mohaghegh Ardabili, Ardabil, Iran} \\\\
{ Ali Behtoei\thanks{Corresponding author, a.behtoei@sci.ikiu.ac.ir}}\\
{\it \small Department of Pure Mathematics, Faculty of Science, Imam Khomeini International University,} \\
 { \it \small Qazvin, Iran, PO Box: 34148 - 96818.}  
}
\date{}
\maketitle
\vspace*{-5mm}
\begin{abstract}
Decomposition into special substructures inheriting significant properties is an important method for the investigation of some mathematical structures.
A \emph{total dominator coloring} (briefly, a TDC)  of a graph $G$ is a proper coloring (i.e. a partition of the vertex set $V(G)$ into independent subsets named color classes)  in which each vertex of the graph is adjacent to all of vertices of some
color class. The \emph{total dominator chromatic number} $\chi_{td}(G)$
of $G$ is the minimum number of color classes in a TDC of $G$.
In this paper among some other results and by using the existance of  Steiner triple systems, we determine the total dominator chromatic number of the Kneser graph $KG(n,2)$ for each $n\geq 5$. 
\end{abstract}



\section{Introduction}

Decomposition into special substructures inheriting significant properties is an important method for the investigation of some mathematical structures,  
 Let $G=(V,E) $ be a graph with the \emph{vertex set} $V$ of \emph{order}
$n(G)$ and the \emph{edge set} $E$ of \emph{size} $m(G)$. The
\emph{open neighborhood} and the \emph{closed neighborhood} of a
vertex $v\in V$ are $N_{G}(v)=\{u\in V\ |\ uv\in E\}$ and
$N_{G}[v]=N_{G}(v)\cup \{v\}$, respectively. The \emph{degree} of a
vertex $v$ is also $deg_G(v)=|N_{G}(v)|$. An isolated vertex is a vertex with degree zero.
A \emph{proper coloring} of a graph $G =(V,E)$ 
 is a function from
the vertices of the graph to a set of colors such that any two
adjacent vertices receive different colors. The \emph{chromatic number}
$\chi (G)$ of $G$ is the minimum number of colors needed in a proper
coloring of a graph. In a proper coloring of a graph, a \emph{color class} of the
coloring is a set consisting of all those vertices with the same
color. If $f$ is a proper coloring of $G$ with the coloring classes
$V_1,~V_2,~ \ldots,~V_{\ell}$ such that every vertex in $V_i$ has
color $i$, we write simply $f=(V_1,V_2,\ldots,V_{\ell})$.
A \emph{total dominating set}, briefly $TDS$, $S$ of a graph $G$ \cite{HeYe13} is a subset
of the vertices in $G$ such that for each vertex $v$, $N_G(v)\cap
S\neq \emptyset$. The \emph{total domination number $\gamma_t(G)$} of $G$ is the minimum cardinality of a $TDS$ of $G$. 
Motivated by the relation between coloring and domination, the notion of total dominator coloring was introduced in \cite{Kaz2015}. Also the reader can consult \cite{Australasian2015,Hen2015,Kaz2016}  for more information.
A \emph {total dominator coloring}, briefly TDC, of a graph $G$ is a proper coloring of $G$ in
which each vertex of the graph is adjacent to every vertex of some
color class. The \emph {total dominator chromatic number} $\chi_{td}(G)$
of $G$ is the minimum number of color classes in a TDC of $G$.
A nice relation between these graphical invariants is provided in \cite{Hen2015}  as follows.
\begin{proposition} \label{bound}
\cite{Hen2015} For each  isolated free graph $G$ we have
\begin{eqnarray*}
\max\{\chi(G),\gamma_t(G)\} \leq \chi_{td}(G)\leq \gamma_t(G)+\chi(G).
\end{eqnarray*}
\end{proposition}

For positive integers $n,k$ with the condition $k\leq\frac{n}{2}$ let $V={[n] \choose k}$ be the set of all $k-$subsets of the $n$-set $\{1,2,\ldots, n\}$. The Kneser graph  $KG(n,k)$, is a graph with the vertex set $V$ such that two vertices are adjacent in it if and only if the corresponding subsets are disjoint. For example $KG(2k,k)$ is a matching and $KG(5,2)$ is the Petersen graph. 
Note that each vertex in $KG(n,2)$ which is a $2-$subset of the set $\{1,2,\ldots, n\}$ corresponds to an edge in the complete graph $K_n$ with vertex set $\{1,2,\ldots,n\}$. Hence, two vertices of $KG(n,2)$ are non-adjacent if and only if the corresponding edges in $K_n$ are adjacent. 
Also, If $A$ is an independent set of vertices in $KG(n,2)$, then either all vertices in $A$ have a common symbol , say $a$, or $A=\{ab,ac,bc\}$, for some $a,b,c\in \{1,2,\ldots n\}$. In other word, an independent set of vertices in $KG(n,2)$ corresponds to a star subgraph with center $a$ (and in this case $"a"$ is considered to be a central symbol) or a triangle subgraph in $K_n$.  From now on, we call the independent set (color class) in $KG(n,2)$ of the first form starlike with center $a$ and the second form triangular. Moreover, for convenient we denote the vertex $\{a,b\}$ with $ab$. When a starlike color class consists of only one vertex, say $ij$, then we choose in arbitray exactly one of these two symbols $i$ or $j$ as the central symbol of this color class. Note that two different starlike color classes may have the same central symbols. Since every proper coloring is a partition of vertices into independent sets of vertices, we can consider every proper coloring of $KG(n,2)$ as an edge decomposition of the complete graph $K_n$ into star and triangle subgraphs.
The total domination number of  the Kneser graphs $KG(n,2)$  is completely determned in \cite{domination} as follows.
\begin{proposition}\label{gammat}
If $n\geq 4$, then
\begin{equation*}
\gamma_t(KG(n,2))= \left\{
\begin{array}{ll}
\! 6 &  ~n=4,\\
\! 4 & ~ n=5,\\
\! 3  &  ~ o.w.
\end{array}
\right.
\end{equation*}
\end{proposition}

Also, it is well known that $\chi(KG(n,k))=n-2k+2$ and hence $\chi(KG(n,2))=n-2$, see \cite{lovasz}. 
In recent years many different types of colorings for the Kneser graphs has been studied by several researchers.
For instance, the $b-$chromatic number of some Kneser graphs is investigated in \cite{omoomi} and the locating chromatic number of  Kneser graphs is studied in \cite{AB} in which the locating chromatic number of  $KG(n,2)$ is completely determined.
In this paper, we determine the total dominator coloring of the  Kneser graph $KG(n,2)$. 


\section{Main results}

Since $\chi(KG(n,2))=n-2$,  two Propositions \ref{bound} and  \ref{gammat} imply the following result. 

\begin{corollary}  \label{n+1} 
For each integer $n\geq 6$, we have 
\begin{eqnarray*}
n-2\leq \chi_{_{td}}(KG(n,2))\leq n+1.
\end{eqnarray*}
\end{corollary}

\begin{lemma}\label{ntdc}
Let $k\geq n-2$ and $f=(V_1, V_2,\ldots, V_k)$ be a proper coloring of the Kneser graph $KG(n,2)$. If there exists a symbol $i\in\{1,2,...,n\}$ such that ''$i$" appears in at least $k-1$ color classes, then $f$ is not a TDC.
\end{lemma}
\begin{proof}{
Since for each $i\in\{1,2,...,n\}$ we have $|\{ij:~i\neq j,~1\leq j\leq n\}|=n-1$, the number of vertices in $KG(n,2)$ which contain the symbol  ''$i$" is $n-1$. 
Hence, each symbol can appear in at most  $n-1$ color classes. First assume that "$i$" appears in all of the color classes and hence, $k\leq n-1$. 
Let $j\in \{1,2,\ldots,n\}\setminus \{i\}$ and note that in each color class there exists (at least) one vertex which contains the symbol $i$ and hence, is not adjacent to the vertex $ij$. This means that the vertex $ij$ can not be adjacent to all of the vertices of a color class and thus $f$ is not a $TDC$.\\
Now, suppose that the symbol $i$ appears in exactly $k-1$ color classes and there exist a color class $V_j$ such that $i$ does not appear in it. Let $i'j'$ be a vertex in $V_j$. Note that in this case the vertex $ii'$ can not be adjacent to all of vertices of a color class which means that $f$ is not a TDC and this completes the proof.
}\end{proof}

In \cite{AB} all of optimal $(n-2)-$colorings of the Kneser graph $KG(n,2)$ are characterized as follows.
\begin{thm}\cite{AB}\label{AB}
In every proper $(n-2)-$coloring of the Kneser graph $KG(n,2)$, $n\geq 5$, there exists a unique triangular color class. Furthermore, if $c$ is a proper $(n-2)-$coloring of $KG(n,2)$, then by renaming the symbols $1,2,\ldots, n$, if it is necessary, we have the color classes $F_1, F_2,\ldots, F_{n-2}$ with the following properties.\\
(a) $F_{n-2}=\{n(n-1), n(n-2), (n-1)(n-2)\}$, i.e. $F_{n-2}$ is triangular;\\
(b) For each $i$, $1\leq i\leq n-3$, $F_i$ is starlike with center $i$, and $\{in, i(n-1), i(n-2)\}\subseteq F_i$;\\
(c) Each vertex $ij$ with $\{i,j\}\cap\{n,n-1,n-2\}=\emptyset$, is either in $F_i$ or $F_j$.
\end{thm} 

Now the following result is obtained.

\begin{corollary} \label{LowerBound}
For each integer $n\geq 5$, we have  $~n-1 \leq \chi_{_{td}}(KG(n,2))$.
\end{corollary}
\begin{proof}{
By Theorem \ref{AB}, in each proper $(n-2)-$coloring of the Kneser graph $KG(n,2)$, there exists a symbol (namely $n$, $n-1$ or $n-2$ by renaming the symbols  if it is necessary) which appears in all of the color classes. Thus, by using Lemma \ref{ntdc} the result follows.
}\end{proof}

\begin{lemma}\label{triangle}
{\rm
Let $f$ be a proper $(n-1)-$coloring of the Kneser graph $KG(n,2)$ such that all of the color classes are starlike. Then, $f$ is not a $TDC$.
}
\end{lemma}
\begin{proof}{
Since there are $n-1$ starlike color classes, there exist at most $n-1$ (distinct) central symbols. Hence, there exists a non-central symbol  $i\in \{1,2,...,n\}$. 
Therefor, the symbol $i$ appears at most once in each color class. Since the number of vertices containing the symbol $i$ is $n-1$,  this symbol should appear exactly once in each of the $n-1$ color classes. Now  Lemma \ref{ntdc} implies that $f$ is not a $TDC$. 
}\end{proof}

\begin{proposition} \label{UpperBound}
The total dominator chromatic number of the Petersen graph is $6$ and for each interger $n\geq 6$ we have $\chi_{_{td}}(KG(n,2))\leq n$. 
\end{proposition}
\begin{proof}{
It is not hard to see by investigation that there does not exist a $TDC$ for the Petersen graph with 3,4 or 5 color classes and in Table \ref{table1} a $TDC$ of $KG(5,2)$ with 6 color classes is presented. Hence, $\chi_{_{td}}(KG(5,2))=6$. Also, it is easy to check that for each $n\in\{5,6,7,8,9\}$  the corresponding vertex partition of $KG(n,2)$ presented in Table \ref{table1} provides a total dominator coloring for  $KG(n,2)$. 
In fact the general model for color classes for each $n\geq 7$ can be stated as follows. 
Assume that $n\geq 7$. Let 
$$V_1=\{12\},~V_2=\{34\},~V_3=\{25,26,56\},~V_4=\{15,16\},~V_5=\{31,32,35,36\},~ V_6=\{41,42,45,46\}$$
and
$$V_i=\{ik:~k\neq i,~1\leq k \leq n\} \setminus \bigcup_{j=1}^{i-1}V_j~:~~7\leq i\leq n.$$
This partition of the vertices of $KG(n,2)$ provides a proper $n$-coloring for $KG(n,2)$ in which the color class $V_3$ is triangular and all other color classes are starlike.
Note that each vertex in $V_1\cup V_2\cup \cdots \cup V_6$ is adjacent to all of vertices of some $V_j$, $j\in \{1,2,...,6\}$. Also, for each $i\geq 7$ the vertex $ij$ is adjacent to at least one of the vertices $12$ or $34$ and this means that the vertex $ij$ is adjacent to all of vertices of $V_1$ or $V_2$. Therefor, this vertex partition provides a $TDC$ for $KG(n,2)$  and hence $\chi_{_{td}}(KG(n,2))\leq n$ for each  $n\geq 7$.
}\end{proof}

\begin{table}[ht] 
\begin{center}
\begin{tabular}{|cl|} \hline 
$KG(5,2)$ : & \{12\},\{34\},\{25\},\{15\},\{13,23,35\},\{14,24,45\}   \\ \hline 
$KG(6,2)$ : &  \{12\},\{34\},\{25,26,56\},\{15,16\},\{31,32,35,36\},\{41,42,45,46\}   \\ \hline 
$KG(7,2)$ : &  \{12\},\{34\},\{25,26,56\},\{15,16\},\{31,32,35,36\},\{41,42,45,46\}, \\ & \{71,72,73,74,75,76\}   \\ \hline 
$KG(8,2)$ : &  \{12\},\{34\},\{25,26,56\},\{15,16\},\{31,32,35,36\},\{41,42,45,46\}, \\ & \{71,72,73,74,75,76,78\},\{81,82,83,84,85,86\}  \\ \hline 
$KG(9,2)$ : &  \{12\},\{34\},\{25,26,56\},\{15,16\},\{31,32,35,36\},\{41,42,45,46\}, 
\\ & \{71,72,73,74,75,76,78,79\},\{81,82,83,84,85,86,89\},\{91,92,93,94,95,96\}  \\ \hline 
\end{tabular}
 \caption{Total dominator coloring for some Kneser graphs.}
 \label{table1}
 \end{center}
\end{table}

\begin{lemma} \label{1t5}
Assume that there exists a total dominator coloring of $KG(n,2)$ with $n-1$ color classes, $n\geq 6$, and let $t$ be the number of triangular color classes. 
Then we have $1\leq t\leq 5$. 
\end{lemma}

\begin{proof}{
Lemma \ref{triangle} implies that $t\geq 1$. Let $\ell$ be the number of starlike color classes and hence, $$\ell=(n-1)-t\leq n-2.$$
 By renaming the symbols $1, 2, \ldots, n$, if it is necessary, we may assume that the centers of these $\ell$ starlike classes (which some of them may have similar centers) are included in the $\ell$-set  $\{n, n-1, \ldots, n-(\ell-1)\}$, and hence we can asuume that these starlike color classes are indexed as $V_n, V_{n-1}, \ldots, V_{n-\ell+1}$. Therefore,  $n-\ell$ symbols $1,2, \ldots, n-\ell$  are not the central symbol of any starlike color class. 
 The number of vertices of $KG(n,2)$ whose both symbols lie in the set $\{1,2,\ldots, n-\ell\}$ is equal to $n-\ell\choose 2$ and all of these vertices must be distributed among triangular color classes. Since each triangular color class contains exactly three vertives, for the distribution of these $n-\ell\choose 2$ vertices at least ${\frac{1}{3}}{{n-\ell} \choose 2}$ triangular color classes is needed. Thus, we must have
$$t\geq {\frac{1}{3}}{{n-\ell} \choose 2} = \frac{(n-\ell)(n-\ell-1)}{6}$$
and the relation  $t=(n-1)-\ell$ implies that $$(n-1)-\ell \geq \frac{(n-\ell)(n-\ell-1)}{6}.$$ 
Since $n-\ell-1\neq 0$, we conclude that $t\leq 5$. 
}\end{proof}

Recall that a balanced incomplete block design (a BIBD) with parameters $t,n,k,\lambda$ (i.e. a $t-(n,k,\lambda)$ design) is an ordered pair $(S,\beta)$ in which $S$ is a set of $n$ points (or symbols) and $\beta$ is a family of $k-$subsets of $S$ called blocks, such that every $t$ elements of $S$ occur together in exactly $\lambda$ blocks of $\beta$.  
When $\lambda=1$ the design is called a Steiner system, and when $k=3$ it is called a triple system. A design with parameters $t=2$, $k=3$ and $\lambda=1$ with $n$ points is called a Steiner triple system of order $n$, denoted by $STS(n)$.
It is well known that a Steiner triple system of order $n$ exists if and only if $n\equiv 1,3\pmod{6}$, see \cite{sts}. 
 
Now we are ready to state the main theorem as follows.


\begin{theorem}
For the total dominator chromatic number of the Kneser graphs $KG(n,2)$ we have 
\begin{equation*}
\chi_{_{td}} (KG(n,2))= \left\{
\begin{array}{ll}
6 & \mbox{ if } n=5,\\
n & \mbox{ if } n\geq 6.
\end{array}
\right.
\end{equation*}
\end{theorem}
\begin{proof}{
By using Corolary \ref{LowerBound} and Proposition \ref{UpperBound},  $\chi_{_{td}}(KG(5,2))=6$ and  
for each $n\geq 6$ we have $n-1 \leq \chi_{_{td}}(KG(n,2)) \leq n$.
Here after we assume that $n\geq 6$ and we want to show that $\chi_{_{td}}(KG(n,2)) \neq n-1 $. 
Suppose on the contrary that there exists a total dominator coloring of $KG(n,2)$ with $n-1$ color classes.
Thus, by Lemma \ref{1t5} and by using the notations appeared in it and in its proof, we have $1\leq t \leq 5$. 
Let $C\subseteq \{1,2,...,n\}$ denotes the set of central symbols and let $\overline{C}=\{1,2,...,n\}\setminus C$. 
By considering the following cases based on different possibilities for $t$, we show that each of them leads to a contradiction, and  this  completes the proof.

\textbf{Case 1:}  $t=5$.\\
In this case we have  $\ell=n-1-t=n-6$ and hence $|C|\leq n-6$. Thus, there are at least $6$ non-central symbols, say $\{1,2,3,4,5,6\}$. 
Hence, $\{1,2,3,4,5,6\}\subseteq \overline{C}$.
Asume that $V_1, V_2, V_3, V_4, V_5$ are the triangular color classes. 
Note that when $i\neq j$ and $\{i,j\}\subseteq\overline{C}$ then non of the starlike color classes may contain the vertex $ij$.
Hence, we should have $$\{ij~|~i\neq j, \{i,j\}\subseteq\overline{C}\}\subseteq\bigcup_{k=1}^5 V_k.$$
Since $\{1,2,3,4,5,6\}\subseteq \overline{C}$   and  $$\bigg{|}\bigg{\{}ij~|~i\neq j , \{i,j\}\subseteq \{1,2,3,4,5,6\}~\bigg{\}}\bigg{|}={6\choose 2} =15=\sum_{k=1}^5 |V_k|,$$
all of $5$  triangular color classes must be constructed just based on the symbols $\{1,2,3,4,5,6\}$. 
For each $i\in \{1,2,3,4,5,6\}$ let $B_i$ be the set of appearing symbols in $V_i$. 
Since each triangular color class $V_i$ is constructed based on three symbols and its vertices are all of three possible pairings of these three symbols, 
we have $|B_i|=3$ for each $1\leq i\leq 6$. Also, note that each pair of distinct symbols (points) $i,j$ with the condition $\{i,j\}\subseteq \{1,2,3,4,5,6\}$, there exists
exactly one ($\lambda=1$) set (block) $B_k$, $1\leq k \leq 6$, such that $\{i,j\}\subseteq B_k$. 
This is equivalent to the existence of a $STS(6)$, which is impossible. 

\textbf{Case 2:}  $t=4$.\\
In this case  $\ell=n-1-t=n-5$ and hence $|C|\leq n-5$. This means that there are $k=|\overline{C}|\geq 5$ non-central symbols.
Assume that $V_1, V_2, V_3, V_4$ are the triangular color classes.
If $k\geq 6$, then $k\choose 2$ vertices $\{ij:~i\neq j,~\{i,j\}\subseteq \overline{C}\}$ should be distributed among these $4$ triangular color classes, 
which is impossible because 
$${k\choose 2}\geq {6\choose 2}=15>12=|V_1\cup V_2\cup V_3\cup V_4|.$$
Thus we must have exactly $k=5$ non-central vertices, say $\{1,2,3,4,5\}$ and hence $\overline{C}=\{1,2,3,4,5\}$. 
Since each $V_s$, $1\leq s \leq 4$, is an independent set in triangular form and 
$$\{ij~|~i\neq j, \{i,j\}\subseteq\overline{C}\} \subseteq\bigcup_{s=1}^4 V_s,~~|\{ij~|~i\neq j , \{i,j\}\subseteq\overline{C}\}|=10,~~\bigg{|}\bigcup_{s=1}^4 V_s\bigg{|}=12,$$ 
there exist two different vertices $x=ij$ and $y=i^{\prime} j^{\prime}$ in $\bigcup_{s=1}^4 V_s$    such that 
$$\{i,i^{\prime}\}\subseteq\overline{C},~~ \{j,j^{\prime}\}\cap \overline{C}=\emptyset.$$
Now we consider the following two subcases.
\\
\textbf{Subcase 2.1.} There exists $1\leq s\leq 4$ such that $\{x,y\}\subseteq V_s$.\\
Without loss of generality, assume that $V_s=\{x,y,12\}$ and $x=1j,~y=2j$ (note that $V_s$ is a triangular color class and hence $j=j^{\prime}$). 
"Three" vertices $13,14,15$ should be distributed among the three remaining triangular color classes. 
We know that  each triangular color class which contains a vertex containing the symbol $"1"$, should contain exactly two vertex containing the symbol $1$. 
Thus, an even number of vertices in $(V_1\cup V_2\cup V_3\cup V_4)\setminus V_s$ contain the symbol $1$, a contradiction.

\textbf{Subcase 2.2.} $x\in V_s$ and $y\in V_{s^{\prime}}$ with $s\neq s^{\prime}$ and $1\leq s, s^{\prime}\leq 4$.\\
In this case, $V_s$ should contain two vertices containing the symbol $j$ and $V_{s^{\prime}}$ should contain two vertices containing the symbol $j^{\prime}$. 
This means that
$$|V_1\cup V_2\cup V_3\cup V_4 \setminus \{ij~|~i\neq j, \{i,j\}\subseteq \overline{C}\}|=4$$
which is a contradiction.

\textbf{Case 3:}  $t=3$.\\
We have $\ell=n-1-t=n-4$ and $|C|\leq n-4$ which implies that there are $k\geq 4$ non-central symbols.
If $k\geq 5$, then $k\choose 2$ vertices  which both of their symbols are non-central, should be distributed among $3$ triangular color classes which is impossible since ${k\choose 2}\geq {5\choose 2}=10$ and there are $9$ vertices in these $3$ triangular color classes. 
Thus we  have exactly $4$ non-central symbols, say $\overline{C}=\{1,2,3,4\}$.
Let $V_1, V_2, V_3$ be these  triangular color classes. 
Since $$\{ij~|~i\neq j, \{i,j\}\subseteq\overline{C}\}\subseteq V_1\cup V_2\cup V_3,~~|\{ij~|~i\neq j , \{i,j\}\subseteq\overline{C}\}|={4\choose 2}=6,~~|V_1\cup V_2\cup V_3|=9,$$
there exist three different vertices $x=ij$, $y=i^{\prime} j^{\prime}$ and $z=i^{\prime\prime}j^{\prime\prime}$ in $V_1\cup V_2\cup V_3$    such that $$\{i,i^{\prime},i^{\prime\prime}\}\subseteq\overline{C},~~\overline{C}\cap \{j,j^{\prime},j^{\prime\prime}\}=\emptyset.$$
We consider three subcases as below.

\textbf{Subcase 3.1.} There exists $1\leq s\leq 3$ such that $\{x,y,z\}\subseteq V_s$.\\
Thus, two remaining triangular color classes must be constructed based on the symbols (points) $\{1,2,3,4\}$ which is impossible, because  $STS(4)$ does not exist. 

\textbf{Subcase 3.2.} There exist $1\leq s,s'\leq 3$ such that $\{x,y\}\subseteq V_s$ and $z\in V_{t'}$.\\
Then, $V_t$ should contain two vertices containing the symbol $j^{\prime\prime}$. This means that
$$| V_1\cup V_2\cup V_3\setminus \{ij~|~i\neq j, \{i,j\}\subseteq \overline{C}\}|\geq 4,$$
which is a contradiction.
 
\textbf{Subcase 3.3.} $x\in V_1$, $y\in V_2$ and $z\in V_3$.\\
Similarly,  $V_1$ should contain two vertices containing the symbol $j$, $V_2$ should contain two vertices containing the symbol $j^{\prime}$ and $V_3$ should contain two vertices containing the symbol $j^{\prime\prime}$. This means that
$$|(V_1\cup V_2\cup V_3\setminus \{ij~|~i\neq j, \{i,j\}\subseteq \overline{C}\}|\geq 6,$$
which is a contradiction.

\textbf{Case 4:}  $t=2$.\\
In this case  $\ell=n-1-t=n-3$ and $|C|\leq n-3$ which implies that $k=|\overline{C}|\geq 3$.
First assume that $k=4$. Since two triangular color classes contain $6$ vertices together and  ${4\choose 2}=6$,  this leads the existence of a $STS(6)$, which is impossible.
If $k\geq 5$, then $k\choose 2$ vertices whose both symbols are non-central, should be distributed among these two triangular color classes which is impossible because ${k\choose 2}\geq {5\choose 2}=10>6$.
Thus, we  have exactly $k=3$ non-central symbols, say $\{1,2,3\}$.
Hence, three vertice $12,13,23$ should be distributed among two triangular color classes (say, $V_1$ and $V_2$) and 
by the Pigenhole principle, for some $i\in\{1,2\}$ we have $|V_i\cap \{12,13,23\}| \geq 2$ and this implies that $V_i=\{12,13,23\}$.
Without loss of generality, we may assume that $V_1=\{12,13,23\}$ and $V_2=\{ab,ac,bc\}$. 
Since $V_1\cap V_2=\emptyset$, we have $|\{1,2,3\}\cap\{a,b,c\}|\leq 1$.  Without loss of generality, assume that $1\notin \{a,b,c\}$. 
Since there are $\ell=n-3$ starlike color classes  with exactly $n-3$ central symbols $\{4,5,...,n\}$, for each $j\in \{4,5,...,n\}$ there exists a unique starlike color class with central symbol $j$ which contains the vertex $1j$. Hence, the symbol "$1$" appears in exactly $n-2$ color classes which contradicts Lemma \ref{ntdc}. 

\textbf{Case 5:} Let $t=1$.\\
Let $V_1$ be the triangular color classes. Since $\ell=n-2$ and $|C|\leq n-2$, there are $k\geq 2$ non-central symbols. Assume that  $\{1,2\}\subseteq \overline{C}$ and hence $V_1$ contains the vertex $12$. Without lose of generality, assume that $V_1=\{12,13,23\}$.  
Let $i\in\{4,5,...,n\}$. Since the number of vertices containing the symbol "$i$" is $n-1$ and there are $n-2$ (starlike) color classes which may contain a vertex with symbol $i$, by the Pigenhole principle there exists a starlike color class with at least two vertices containing the symbol $i$, i.e a starlike with central symbol $i$. 
Hence, $\{4,5,...,n\}\subseteq C$ and for each $i\in\{4,5,...,n\}$ the vertex $1i$ should be in a starlike color class with central symbol $i$. Therefore, the symbol $1$ appears in at least $n-2$ color classes and this contradicts Lemma \ref{ntdc}.
}\end{proof}
~\\ \\ 
The authors declare that they have no competing interests.



\end{document}